\documentclass[10pt]{article}
\usepackage[all]{xy}
\usepackage{amsfonts,amsmath,oldgerm,amssymb,amscd}
\newcommand{\ra}{\rightarrow}		
\newcommand{\by}[1]{\stackrel{#1}{\ra}}

\newcommand{\ol}{\overline}		\newcommand{\wt}{\widetilde}
\newcommand{\iso}{\by \sim}              

\newtheorem{theorem}{Theorem}[section]
\newtheorem{proposition}[theorem]{Proposition}
\newtheorem{lemma}[theorem]{Lemma}

\newtheorem{corollary}[theorem]{Corollary}
\newtheorem{question}[theorem]{Question}

\newcommand{\ga}{\alpha}	
		
\newcommand{\gf}{\varphi}

	\newcommand{\BF}{\mbox{$\mathbb F$}}

	\newcommand{\BZ}{\mbox{$\mathbb Z$}}

\newcommand{\CM}{\mbox{$\mathcal M$}}

\newcommand{\ot}{\mbox{\,$\otimes$\,}}	\newcommand{\op}{\mbox{$\oplus$}}

\newcommand{\Aut}{\mbox{\rm Aut\,}}	
\newcommand{\Hom}{\mbox{\rm Hom\,}}	

\newcommand{\Um}{\mbox{\rm Um}}		
\newcommand{\GL}{\mbox{\rm GL}}

\begin{document}

\begin{center}
{{\bf \large Cancellation problem for projective modules over affine
algebras }\\\vspace{.2in} 
        {\large Manoj  Kumar Keshari}\\
\vspace{.1in}
{\small
Department of Mathematics, IIT Mumbai, Mumbai - 400076, India;\;
        keshari@math.iitb.ac.in}} 
\end{center}

\section{Introduction}
All the rings are assumed to be commutative Noetherian and all the
modules are finitely generated.

Let $A$ be a ring of dimension $d$ and let $P$ be a projective
$A$-module of rank $n$. We say that $P$ is {\it cancellative} if $P\op
A^m\iso Q\op A^m$ for some projective $A$-module $Q$ implies
$P\iso Q$.

A classical result of Bass (\ref{bass}) says that if $n>d$, then $P$
is cancellative. It is well known that Bass' result is best possible
in general (since tangent bundle of real $2$-sphere is stably trivial
but not trivial). However, Bass' result can be improved in some
specific cases which we describe below.
%we have the following results in the case $n=d$.

\begin{theorem}\label{z1}
$(i)$ If $A$ is an affine algebra of dimension $d$ over an algebraically
closed field, then Suslin \cite{suslin1} proved that every projective
$A$-module of rank $\geq d$ is cancellative. 

$(ii)$ If $A$ is an affine algebra of dimension $d$ over an infinite
perfect $C_1$-field $k$ such that $1/d! \in k$, then Suslin \cite{suslin}
proved that $A^d$ is cancellative. Subsequently, Bhatwadekar
(\cite{Bhh}, Theorem 4.1) proved that every projective $A$-module of
rank $d$ is cancellative.

$(iii)$ If $A$ is an affine algebra of dimension $d$ over $\BZ$, then
Vaserstein (\cite{su}, Corollary 18.1, Theorem 18.2) proved that $A^d$
is cancellative.  Subsequently, Mohan Kumar, Murthy and Roy
(\cite{mmr}, Corollary 2.5) proved that every projective $A$-module of
rank $d$ is cancellative.
\end{theorem}

We note that Bhatwadekar's proof \cite{Bhh} uses Suslin's
result \cite{suslin} that $A^d$ is cancellative. Similarly, the
proof of Mohan Kumar et. al. \cite{mmr} uses Vaserstein's
results \cite{su}. Hence, in view of the above results, we can ask
the following:

\begin{question}\label{q1}
Let $A$ be a ring of dimension $d$. Assume that $A^d$ is
cancellative. Is every projective $A$-module of rank $d$ cancellative?
\end{question}

In (\cite{bh01}, Example 3.11), Bhatwadekar has given an example of a
smooth real affine surface $A$ such that $A^2$ is cancellative, but
$K_A\op A$ is not cancellative, where $K_A$ is the canonical module of
$A$. Thus, the above question has negative answer in general. We will
modify the above question and prove the following result (\ref{q5}).

\begin{theorem}\label{m1}
Let $A$ be a ring of dimension $d$. Assume that for every finite
extension $R$ of $A$, $R^d$ is cancellative. Then every projective
$A$-module of rank $d$ is cancellative.
\end{theorem}

For a ring $k$, a finite extension of an affine $k$-algebra is an
affine $k$-algebra. Hence, in (\ref{z1}), assuming the result of
Suslin \cite{suslin}, our result gives an alternative proof of
Bhatwadekar's result \cite{Bhh}.  Similarly, assuming the result of
Vaserstein \cite{su}, it gives an alternative proof of Mohan Kumar
et. al. \cite{mmr}.

Regarding question (\ref{q1}), Bhatwadekar (\cite{bh01}, Proposition
3.7) proved the following interesting result: Let $A$ be a ring of
dimension $2$ and let $P$ be a projective $A$-module of rank $2$. If
$\wedge^2(P)\op A$ is cancellative, then $P$ is cancellative. In
particular, if $A^2$ is cancellative, then every projective $A$-module
of rank $2$ with trivial determinant is cancellative. In view of this
result, Bhatwadekar (\cite{bh00}, Question VII) asked the following
question which is open for $d\geq 3$.

\begin{question}
Let $A$ be a ring of dimension $d$. Assume that $A^d$ is
cancellative. Is every projective $A$-module of rank $d$ with trivial
determinant cancellative?
\end{question}

In \cite{mk}, Mohan Kumar has given an example of a smooth affine algebra of
dimension $n\geq 4$ over which there exist projective modules of rank
$n-2$ that are not cancellative. More precisely, he proved the
following: let $p$ be a prime integer and let $k$ be any algebraically
closed field. Then there exists an $f\in A=k[X_1,\ldots,X_{p+2}]$ and
a projective $A_f$-module $P$ of rank $p$ such that $P\op A_f\iso
A_f^{p+1}$ but $P\not\iso A_f^p$, i.e. $P$ is not cancellative.

In view of the above results, the only case remaining regarding
cancellation problem is when rank $P=\dim A -1$.

\begin{question}\label{q2}
Let $A$ be an affine algebra of dimension $n\geq 3$ over an
algebraically closed field $k$. Let $P$ be a projective $A$-module of
rank $n-1$. Is $P$ cancellative?
\end{question}

This is not known even when $n=3$ and $P=A^2$. We prove the following
result (\ref{d-1}) which is analogue of (\ref{m1}) for affine algebras
over $\ol\BF_p$ .

\begin{theorem}
Let $A$ be an affine algebra of dimension $d\geq 4$ over
$\ol\BF_p$. Assume that if $R$ is a finite extension of $A$, then
$R^{d-1}$ is cancellative. Then every projective $A$-module of rank $d-1$
is cancellative. 
\end{theorem}

Let $R$ be an affine algebra of dimension $d-1$ over an algebraically
closed field $k$ with $1/(d-1)! \in R$. Then Wiemers (\ref{w002})
proved that projective $R[X]$-modules of rank $d-1$ are cancellative,
thus answering question (\ref{q2}) in affirmative in the polynomial
ring case. We prove the following two results (\ref{1.3}) and
(\ref{z9}, \ref{z11}) which answers question (\ref{q2}) in affirmative
in some special cases.

\begin{theorem}
Let $k$ be an algebraically closed field with $1/d! \in k$ and let $R$
be an affine $k$-algebra of dimension $d$. Assume that $f(T)\in R[T]$ is
a monic polynomial and either

$(i)$ $A=R[T,1/f]$ or 

$(ii)$ $A=R[T,f_1/f,\ldots,f_r/f]$, where $f,f_1,\ldots,f_r\in R[T]$
is a regular sequence.

Then $A^d$ is cancellative.
\end{theorem}

\begin{theorem}
Let $R$ be an affine algebra of dimension $d$ over an algebraically
closed field $k$ with $1/d! \in k$. Let $P$ be a projective
$R[X,X^{-1}]$-module of rank $d$. Then 

$(i)$ $P$ is cancellative and 

$(ii)$ the natural map $\Aut(P) \ra \Aut(P/(X-1)P)$ is surjective.
\end{theorem}

We also prove the analogue of above results for affine algebras over
real closed fields (\ref{z6}, \ref{2.2}, \ref{free1},
\ref{free2}). Note that $(ii)$ extends our earlier result
(\cite{kes05}, Theorem 4.7), where it is proved for projective
$A$-modules which are extended from $R$ under the assumption that $R$
is smooth.

\begin{theorem}
Let $k$ be a real closed field and let $R$ be an affine $k$-algebra of
dimension $d-1 \geq 2$. Assume that $f(T)\in R[T]$ is a monic polynomial which
does not belong to any real maximal ideal. Then the following holds:

$(i)$ If $A=R[T,1/f]$, then every projective $A$-module of rank $d$ is
cancellative.

$(ii)$ If $A=R[T,f_1/f,\ldots,f_r/f]$, where $f,f_1,\ldots,f_r$ is a
$R[T]$-regular sequence, then every projective $A$-module of rank $d$
with trivial determinant is cancellative.

$(iii)$ Further, if $R=B[X]$, then $A^{d-1}$ is also cancellative in $(i, ii)$.
\end{theorem}

%######################################################33
%$$$$$$$$$$$$$$$$$$$$$$$$$$$$$$$$$$$$$$$$$$$$$$$$$$$$$$

\section{Preliminaries}

Let $B$ be a ring and let $P$ be a projective $B$-module. Recall that
$p\in P$ is called a {\it unimodular element} if there exists a $\psi
\in P^*=\Hom_B(P,B)$ such that $\psi(p)=1$. We denote by $\Um(P)$, the
set of all unimodular elements of $P$. We write $O(p)$ for the ideal
of $B$ generated by $\psi(p)$, for all $\psi \in P^*$. Note that, if
$p\in \Um(P)$, then $O(p)=B$. For an ideal $J\subset B$, we denote by
$\Um^1(B\op P,J)$, the set of all $(a,p)\in\Um(B\op P)$ such that
$a\in 1+J$ and by $\Um(B\op P,J)$, the set of all $(a,p)\in\Um^1(B\op
P,J)$ such that $p\in JP$.
  
Given an element $\gf\in P^\ast$ and an element $p\in P$, we define an
endomorphism $\gf_p$ of $P$ as the composite $P\by \gf B\by p P$.  If
$\gf(p)=0$, then ${\gf_p}^2=0$ and hence $1+\gf_p$ is a uni-potent
automorphism of $P$.
By a {\it transvection}, we mean an automorphism of $P$ of the form
$1+\gf_p$, where $\gf(p)=0$ and either $\gf \in \Um(P^\ast)$
or $p \in\Um(P)$. We denote by $E(P)$, the subgroup of
$\Aut(P)$ generated by all transvections of $P$. Note that $E(P)$ is a
normal subgroup of $\Aut(P)$. 

An existence of a transvection of $P$ pre-supposes that $P$ has a
unimodular element. Let $P = B\op Q$, $q\in Q, \alpha\in
Q^*$. Then the automorphisms $\Delta_q$ and $\Gamma_\ga$ of $P$ defined
by $\Delta_q(b,q')=(b,q'+bq)$ and
$\Gamma_\alpha(b,q')=(b+\alpha(q'),q')$ are transvections of
$P$. Conversely, any transvection $\Theta$ of $P$ gives rise to a
decomposition  $P=B\op Q$ in such a way that $\Theta = \Delta_q$ or
$\Theta = \Gamma_\alpha$.

For an ideal $J\subset B$, we denote by $EL^1(B\op P,J)$,
the subgroup of $E(B\op P)$ generated by
$\Delta_q$ and $\Gamma_{a\phi}$, where $q\in P, a\in J, \phi\in P^*$.
\medskip

We begin by stating two classical results due to Serre \cite{serre} and
Bass \cite{Bas}.

\begin{theorem}\label{serre1}
Let $A$ be a ring of dimension $d$. Then any projective $A$-module of
rank $>d$ has a unimodular element. In particular, if $\dim A=1$, then
any projective $A$-module of trivial determinant is free.
\end{theorem}

\begin{theorem}\label{bass}
Let $A$ be a ring of dimension $d$ and let $P$ be a projective
$A$-module of rank $>d$. Then $E(A\op P)$ acts transitively on
$\Um(A\op P)$. In particular, $P$ is cancellative.
\end{theorem}

The following two results are due to Lindel (\cite{lindel}, Theorem
2.6 and Lemma 1.1).

\begin{theorem}\label{lin}
Let $A$ be a ring of dimension $d$ and $R=A[X_1,\ldots,X_n,Y_1^{\pm
1},\ldots, Y_m^{\pm 1}]$. Let $P$ be a projective $R$-module of rank
$\geq$ max $(2,d+1)$. Then $E(R\op P)$ acts transitively on $\Um(R\op
P)$. In particular, projective $R$-modules of rank $>d$ are
cancellative.
\end{theorem}

\begin{lemma}\label{3.1}
Let $A$ be a  ring and let $P$ be a projective $A$-module of rank $r$.
Then there exists $s\in A$ such that the following holds:

$(i)$ $P_s$ is free, 

$(ii)$ there exists $p_1,\ldots,p_r\in P$ and
$\phi_1,\ldots,\phi_r\in \Hom(P,A)$ such that
$(\phi_i(p_j))=$ diagonal $(s,\ldots,s)$,

$(iii)$ $sP\subset p_1A+\ldots+ p_rA$,

$(iv)$ the image of $s$ in $A_{red}$ is a non-zero-divisor and

$(v)$ $(0:sA)=(0:s^2A)$.
\end{lemma}

The following result is due to Bhatwadekar and Roy (\cite{Bh-Roy},
Proposition 4.1).

\begin{proposition}\label{trans}
Let $A$ be a ring and let $I$ be an ideal of $A$. Let $P$ be a projective
$A$-module. Then any transvection of $P/IP$ can be lifted to an
automorphism of $P$.
\end{proposition}

\begin{define}
For a ring $A$, we say that projective stable range of $A$ is $\leq r$
(notation: psr$(A)\leq r$) if for all projective $A$-modules $P$ of rank
$\geq r$ and $(a,p)\in \Um(A\op P)$, we can find  $q\in P$ such that
$p+aq\in \Um(P)$. Similarly, $A$ has stable range
$\leq r$ (notation: sr$(A)\leq r$) is defined the same way as
psr$(A)$ but with $P$ required to be free. 
\end{define}
\medskip

The following two results are due to Suslin, Vaserstein (\cite{su},
Corollary 17.3) and Mohan Kumar, Murthy, Roy (\cite{mmr}, Theorem 3.7)
respectively.

\begin{theorem}\label{svv}
Let $k\subset \ol\BF_p$ be a field and let $A$ be an affine
$k$-algebra of dimension $d$. Then sr$(A) \leq $ max $(2,d)$. In
particular, $A^d$ is cancellative.
\end{theorem}

\begin{theorem}\label{kumar}
Let $A$ be an affine algebra of dimension $d\geq 2$ over $\ol
\BF_p$. Suppose that $A$ is regular when $d=2$. Then psr $(A) \leq d$.
\end{theorem}

The following result is due to Quillen \cite{quillen} and Suslin
\cite{susl}.

\begin{theorem}\label{monic}
Let $A$ be a ring and let $P$ be a projective $A[T]$-module.  Assume
that $P_f$ is free for some monic polynomial $f\in A[T]$.  Then $P$ is
free.
\end{theorem}

The following result is due to Wiemers (\cite{Wiemers}, Theorem 3.2,
Corollary 3.4).

\begin{proposition}\label{3.2}
Let $R$ be a ring of dimension $d$ and $A=R[X_1,\ldots,X_n,Y_1^{\pm
1},\ldots,Y_m^{\pm 1}]$.  Let $P$ be a projective $A$-module of rank
$\geq$ max $(2,d+1)$. Then 

$(i)$ $EL^1(A\op P,Y_m-1)$ acts transitively on $\Um^1(A\op P,Y_m-1)$.

$(ii)$ the natural map $\Aut_A(P) \ra \Aut_{A/(Y_m-1)}(P/(Y_m-1)P)$ is
surjective.
\end{proposition}

We state two results due to Wiemers (\cite{Wiemers}, Lemma 4.2 and
Theorem 4.3) respectively which are very crucial for our results.

\begin{proposition}\label{w001}
Let $A$ be a ring and let $P$ be an $A$-module (need not be
projective). Assume that there exists $p=[p_1,\ldots,p_n] \in
\Hom_A(A^n,P)$, $\phi=[\phi_1,\ldots,\phi_n]^t\in \Hom_A(P,A^n)$ and
$s_1,\ldots,s_n\in A$ such that 

$(i)$ $(0:s_i)=(0:s_i^2)$ for $i=1,\ldots,n$,

$(ii)$ $(\phi_i(p_j))_{n\times n} =$ diagonal $(s_1,\ldots,s_n)=N$.

Let $\CM$ be the subgroup of $\GL_n(A)$ consisting of all matrices
$1_n+T.N^2$ for some matrix $T$. Then the map $\Phi : \CM \ra
\Aut_A(P),$ defined by $\Phi(1_n+T.N^2)= Id_P+p.T.N.\phi$ is a group
homomorphism.
\end{proposition}

\begin{theorem}\label{w002}
Let $R$ be a ring of dimension $d$ with $1/d!\in R$ and 
$A=R[X_1,\ldots,X_n,Y_1^{\pm 1},\ldots, Y_m^{\pm 1}]$. Let $P$ be a
projective $A$-module of rank $d$. If $P/(X_1,\ldots,X_n)P$ is
cancellative, then $P$ is cancellative. In particular, if projective
$R$-modules of rank $d$ are cancellative, then projective
$R[X_1,\ldots,X_n]$-modules of rank $d$ are also cancellative.
\end{theorem}

We end this section by stating two results due to Keshari
(\cite{kes05}, Theorem 3.5 and Theorem 4.4) and (\cite{kes}, Theorem
3.10).

\begin{theorem}\label{kes3.5}
Let $R$ be an affine algebra of dimension $n\geq 3$ over an
algebraically closed field $k$ with $1/(n-1)!\in k$.  Let
$g,f_1,\ldots,f_r$ be a $R$-regular sequence and
$A=R[f_1/g,\ldots,f_r/g]$. Let $P'$ be a projective $A$-module of rank
$n-1$ which is extended from $R$. Let $(a,p)\in \Um(A\op P')$ and
$P=A\op P'/(a,p)A$. Then $P$ is extended from $R$.
\end{theorem}

\begin{theorem}\label{kes4.4}
Let $R$ be an affine $k$-algebra of dimension $n\geq 3$, where $k$ is
a real closed field. Let $f\in R$ be an element not belonging to any real
maximal ideal of $A$. Assume that either

$(i)$ $A=R[f_1/f,\ldots,f_r/f]$, where $f,f_1,\ldots,f_r$ is a regular
sequence in $R$ or

$(ii)$ $A=R_f$.

Let $P'$ be a projective $A$-module of rank
$\geq n-1$ which is extended from $R$. Let $(a,p)\in \Um(A\op P')$ and
$P=A\op P'/(a,p)A$. Then $P$ is extended from $R$.
\end{theorem}

%@@@@@@@@@@@@@@@@@@@@@@@@@@@@@@@@@@$$$$$$$$$$$$$$$$$$$$$$$$$$$$$$4
%^^^^^^^^^^^^^^^^^^^^^^^^^^^^^^^^^^^^^^^^^^^^^^^^^^^^^^^^^^^^^^

\section{Main Theorem}

We begin this section with the following result which is very crucial
for later use and seems to be well known to experts. Since we are
unable to find an appropriate reference, we give the complete proof.

\begin{proposition}\label{cartesian1}
Let $A$ be a ring of dimension $d$ and let $J$ be an ideal of
$A$. Consider the cartesian square
$$\xymatrix{ C\ar [r]^{i_1} \ar [d]_{i_2} & A
\ar [d]^{j_1} \\ A \ar [r]_{j_2} & A/J }$$
Then $C$ is finitely generated algebra over $A$ of dimension $d$.
In particular, if $A$ is an affine algebra over a field $k$, then $C$ is 
also an affine algebra over $k$. 
\end{proposition}

\begin{proof}
Recall that $C$ is the subalgebra of $A\times A$ consisting of all
elements $(a,b)$ such that $a-b \in J$. First we will show that $C\iso
A\op J$, where $A\op J$ has the obvious ring structure,
i.e. $(a,x)+(a',x') =(a+a',x+x')$ and
$(a,x).(a',x')=(aa',ax'+a'x+xx')$ for $(a,x),(a',x')\in A\op J$.

We define $i_1:A\op J \ra A$ by $i_1(a,x)=a+x$ and $i_2:A\op J \ra A$ by
$i_2(a,x)=a$. Then $ j_1i_1=j_2i_2$. It is enough to show that $A\op
J$ satisfies the universal property of cartesian square.  Let $B$ be a
ring and let $f_i : B\ra A$ be ring homomorphism, $i=1,2$ such that
$j_1f_1=j_2f_2$. To show that there exists a unique
ring homomorphism $F:B\ra
A\op J$ such that $i_1F=f_1$ and $i_2F=f_2$.

Define $F(b)=(f_2(b), f_1(b)-f_2(b))$. Since $j_1f_1=j_2f_2$, $F:B\ra
A\op J$. Also it is clear that $i_1F=f_1$ and $i_2F=f_2$. It remains
to show that $F$ is a ring homomorphism. Clearly, $F(b+b')=F(b)+F(b')$
for $b,b'\in B$. We have \\

$F(b).F(b')
=(f_2(b),f_1(b)-f_2(b)).(f_2(b'),f_1(b')-f_2(b'))$

$= (f_2(b)f_2(b'), f_2(b)(f_1(b')-f_2(b'))+
f_2(b')(f_1(b)-f_2(b))+(f_1(b)-f_2(b))(f_1(b')-f_2(b')))$

$=(f_2(bb'),f_1(bb')-f_2(bb')) =F(bb').$ \\

Uniqueness of $F$ follows from the fact that $i_1F=f_1$ and $i_2F=f_2$.
This proves that $C\iso A\op J$. If $J=(a_1,\ldots,a_r)$, then $A\op
J$ is generated by $(0,a_1),\ldots,(0,a_r)$ over $ A\op 0$,
since if $x=a_1x_1+\ldots+a_rx_r \in J$, then
$(0,x)=(x_1,0).(0,a_1)+\ldots+(x_r,0).(0,a_r)$. Hence $A\op J$ is a
finitely generated algebra over $A$.

To show that $\dim A\op J =\dim A$, we show that $A\op J$ is integral
over $A$. It is enough to show that $(0,a_i)$, $i=1,\ldots,r$ are
integral over $A$. Clearly $(0,a_i)^2-(a_i,0)(0,a_i)=(0,0)$. This
proves the result.
$\hfill \square$
\end{proof}

\begin{corollary}\label{cartesian}
Let $A$ be a ring and let $s\in A$. Then the cartesian square of $(A,A)$
over $A/sA$ is $A[X]/(X^2-sX)$.
\end{corollary}

The following result is very crucial for later use.

\begin{lemma}\label{q3}
Let $A$ be a ring and let $P$ be a projective $A$-module of rank
$r$. Choose $s\in A$ satisfying the properties of (\ref{3.1}). Assume
that $R^r$ is cancellative, where $R=A[X]/(X^2-s^2X)$. Then $\Aut(A\op
P,sA)$ acts transitively on $\Um^1(A\op P,s^2 A)$.
\end{lemma}

\begin{proof}
Without loss of generality, we can assume that $A$ is reduced.  By
$(\ref{3.1})$, there exist $p_1,\ldots,p_r\in P$ and
$\phi_1,\ldots,\phi_r \in \Hom(P,A)$ such that $P_s$ is free,
$(\phi_i(p_j))=$ diagonal $(s,\ldots,s)$, $sP\subset p_1A+\ldots+
p_rA$ and $s$ is a non-zero-divisor.

Let $(f,q)\in \Um^1(A\op P,s^2 A)$.  Since $f\in 1+s^2A$, by adding
some multiple of $f$ to $q$, we may assume that $q\in s^3P$.
Since $sP \subset p_1A+\ldots+p_rA$, we can write $q=f_1p_1+\ldots
+f_rp_r$ for some $f_i \in s^2A$, $i=1,\ldots,r$.  Note that
$(f,f_1,\ldots,f_r)\in \Um_{r+1}(A,s^2A)$.

By (\ref{cartesian}), $R$ is the cartesian square of $(A,A)$ over $A/s^2A$.
$$\xymatrix{ R\ar [r]^{i_1} \ar [d]_{i_2} & A
\ar [d]^{j_1} \\ A \ar [r]_{j_2} & A/(s^2) }$$

Patching unimodular rows $(f,f_1,\ldots,f_r)$ and
$(1,0,\ldots,0)$ over $A/s^2A$, we get a unimodular row
$(c_0,c_1,\ldots,c_r)\in \Um_{r+1}(R)$. Since $R^{r}$ is cancellative,
there exists $\Theta \in \GL_{r+1}(R)$ such that
$(c_0,c_1,\ldots,c_r)\Theta=(1,0,\ldots,0)$. The projections of this
equation gives
$$(f,f_1,\ldots,f_r)\Psi=(1,0,\ldots,0),\;
(1,0,\ldots,0)\wt \Psi=(1,0,\ldots,0)$$ 
for certain matrices
$\Psi,\wt \Psi\in \GL_{r+1}(A)$ such that $\Psi=\wt\Psi$ modulo
$(s^2)$.  Hence $(f,f_1,\ldots,f_r)\Psi\,\wt\Psi^{-1}=(1,0,\ldots,0)$,
where $\Psi\wt\Psi^{-1}=\Delta\in \GL_{r+1}(A,s^2A)$.

Let $\Delta=1+TN^2$, where $T$ is some matrix and $N=$ diagonal
$(1,s,\ldots,s)$. Applying $(\ref{w001})$ with $n=r+1$ and
$(s_1,\ldots,s_n)=(1,s,\ldots,s)$, we get $\Phi(\Delta) =Id +pTN\phi
\in \Aut (A\op P, sA)$, where $p=[p_1,\ldots,p_n] \in \Hom(A^n,P)$ and
$\phi=[\phi_1,\ldots,\phi_n]^t \in \Hom(P,A^n)$ with $(\phi_i(p_j))=
N= $ diagonal $(1,s,\ldots,s)$. We have \\

$\Phi(\Delta)(f,f_1p_1+\ldots+ f_rp_r)= (Id
+pTN\phi)(f,f_1p_1+\ldots+ f_rp_r)$

$=(f,f_1p_1+\ldots+ f_rp_r)+pTN(f,f_1s,\ldots,f_rs)^t$

$ =
p(f,f_1,\ldots, f_r)^t+pT(f_0,f_1s^2,\ldots,f_rs^2)^t$

$=p(1+TN^2)(f,f_1,\ldots,f_r)^t = p(1,0,\ldots,0)^t = (1,0).$\\
This proves the result. $\hfill \square$
\end{proof}

\begin{corollary}\label{q4}
Let $A$ be a ring of dimension $d$ and let $P$ be a projective
$A$-module of rank $d$. Choose $s\in A$ satisfying the properties of
(\ref{3.1}). Assume that $R^r$ is cancellative, where
$R=A[X]/(X^2-s^2X)$. Then $P$ is cancellative.
\end{corollary}

\begin{proof}
We may assume that $A$ is reduced. By (\ref{bass}), $A\op P$ is
cancellative, hence, we need to show that $\Aut(A\op P)$ acts
transitively on $\Um(A\op P)$. Let $(f,q)\in \Um(A\op P)$. Since $s$
is a non-zero-divisor, $\dim A/s^2 < \dim A$. Hence, by (\ref{bass}),
there exists $\theta \in E(\ol A\op \ol P)$ such that $\theta (\ol
f,\ol q)=(1,0)$, where ``bar'' denotes reduction modulo $(s^2)$. By
(\ref{trans}), $\theta$ can be lifted to $\Theta\in \Aut(A\op P)$ and
$\Theta(f,q)\in \Um^1(A\op P, s^2A)$.  By (\ref{q3}), there exists
$\Theta_1\in \Aut(A\op P)$ such that $\Theta_1\Theta(f,q)=(1,0)$. This
proves the result.  $\hfill \square$
\end{proof}

As a consequence of above result, we prove our first main result.

\begin{theorem}\label{q5}
Let $A$ be a ring of dimension $d$. Assume that for every finite
extension $R$ of $A$, $R^d$ is cancellative. Then every projective
$A$-module of rank $d$ is cancellative.
\end{theorem}

\begin{proof}
Let $P$ be a projective $A$-module of rank $d$. Choose $s\in A$
satisfying the properties of $(\ref{3.1})$. If $R=A[X]/(X^2-s^2X)$,
then $R$ is finite extension of $A$ and hence $R^d$ is
cancellative. By (\ref{q4}), $P$ is cancellative. $\hfill \square$
\end{proof}

\begin{lemma}\label{q6}
Let $A$ be an affine algebra of dimension $d\geq 4$ over $\ol
\BF_p$. Let $P$ be a projective $A$-module of rank $d-1$. Choose $s\in
A$ satisfying the properties of $(\ref{3.1})$. Assume that $R^{d-1}$
is cancellative, where $R=A[X]/(X^2-s^2X)$. Then $P$ is cancellative.
\end{lemma}

\begin{proof}
We can assume that $A$ is reduced and hence $s$ is a non-zero-divisor.
Since, by Suslin's result (\ref{z1}), every projective $A$-module
of rank $d$ is cancellative, it is enough to show that $\Aut(A\op P)$
acts transitively on $\Um(A\op P)$.

Let $(a,p)\in \Um(A\op P )$.  Let ``bar'' denotes reduction modulo
$s^2A$. Then $\dim \ol A=d-1 \geq 3$.  By (\ref{kumar}), there exists $q\in
P$ such that $\ol p+\ol{aq}\in \Um(\ol P)$. Hence there exists $\ol
\sigma \in E(\ol A\op \ol P)$ such that $\ol \sigma(\ol a,\ol
p)=(1,0)$. Lifting $\ol \sigma$ to $\sigma \in \Aut (A\op P)$ and
replacing $(a,p)$ by $\sigma(a,p)$, we may assume that $(a,p)\in
\Um^1(A\op P,s^2A)$.  By (\ref{q3}), there exists $\Delta \in \Aut
(A\op P)$ such that $\Delta(a,p)=(1,0)$. This proves the result.
$\hfill \square$
\end{proof}

\begin{theorem}\label{d-1}
Let $A$ be an affine algebra of dimension $d\geq 4$ over $\ol
\BF_p$. Assume that if $R$ is a finite extension of $A$, then
$R^{d-1}$ is cancellative. Then every projective $A$-module of rank
$d-1$ is cancellative.
\end{theorem}

\begin{proof}
Let $P$ be a projective $A$-module of rank $d-1$. Choose $s\in A$
satisfying the properties of (\ref{3.1}). Since $R=A[X]/(X^2-s^2X)$ is
a finite extension of $A$, $R^{d-1}$ is cancellative, by
hypothesis. Applying (\ref{q6}), $P$ is cancellative. $\hfill \square$
\end{proof}

\begin{proposition}\label{done}
Let $A$ be a ring and let $P$ be a projective $A$-module of rank
$r$. Choose $s\in A$ satisfying the properties of (\ref{3.1}). If
$\GL_{r+1}(A,s^2A)$ acts transitively on $\Um_{r+1}(A,s^2A)$,
then $\Aut(A\op P,sA)$ acts transitively on $\Um^1(A\op P,s^2A)$.
\end{proposition}

\begin{proof}
Let $(a,p)\in \Um^1(A\op P,s^2A)$.  Since $a=1$ modulo $s^2A$, adding
some multiple of $a$ to $p$, we may assume that $p\in s^3P$.  Since,
by (\ref{3.1}), $sP\subset p_1A+\ldots+p_rA$, we get
$p=a_1p_1+\ldots+a_r p_r$ for some $a_i\in s^2A$, $i=1,\ldots,r$.
Note that $(a,a_1,\ldots,a_r)\in \Um_{r+1}(A,s^2A)$.  By assumption,
there exists $\Delta\in \GL_{r+1}(A,s^2A)$ such that $\Delta
(a,a_1,\ldots,a_r)=(1,0,\ldots,0)$.

Let $\Delta=1+TN^2$, where $T$ is some matrix and $N=$ diagonal
$(1,s,\ldots,s)$. Applying $(\ref{w001})$ with $n=r+1$ and
$(s_1,\ldots,s_n)=(1,s,\ldots,s)$, we get $\Phi(\Delta) =Id + pTN
\phi\in \Aut (A\op P)$. It is easy to see that $\Phi(\Delta)\in \Aut
(A\op P,sA)$. Further, as in the proof of (\ref{q3}), we can see that
$\Phi(\Delta)(a,p)=(1,0)$. This proves the result.  $\hfill \square$
\end{proof}

The following result generalizes (\cite{su}, Corollary 17.3).

\begin{theorem}\label{a1}
Let $k\subset \ol \BF_p$ be a field and let $A$ be an affine algebra
over $k$ of dimension $d \geq 2$. Then every projective $A$-module of
rank $d$ is cancellative.
\end{theorem}

\begin{proof}
By Suslin-Vaserstein result (\cite{su}, Corollary 17.3), sr$(A)\leq
d$. Hence every stably free $A$-module of rank $d$ is free, i.e. $A^d$
is cancellative. If $B$ is a finite extension of $A$, then $B$ is also
affine $k$-algebra of dimension $d$ and hence $B^d$ is also
cancellative. By (\ref{q5}), the result follows.
$\hfill \square$
\end{proof}

\begin{remark}\label{c1}
Let $k$ be a field and let $A$ be an affine $k$-algebra of dimension
$d$. Assume that characteristic of $k$ is either $0$ or $p>d$. Further
assume that $cd (k)\leq 1$, where ``cd'' stands for cohomological
dimension \cite{Serre68}. Then $A^{d}$ is cancellative (Suslin's
result). The proof of this result is contained in \cite{suslin} (see
\cite{murthy00}, 2.1 - 2.4).

In particular, if $A$ is an affine $k$-algebra of dimension $d$, where
$k$ is a $C_1$-field of characteristic $0$ or $p>d$. Then $A^d$ is
cancellative. Note that we do not need $k$ to be perfect in (\ref{z1}
(ii)). By (\ref{q4}), we get Bhatwadekar's result (\ref{z1}(ii)) that
every projective $A$-module of rank $d$ is cancellative.
\end{remark}

%$$$$$$$$$$$$$$$$$$$$$$$$$$$$$$$$$$$$$$$$$$$$$$$$$$$$$$$$$$
%$$$$$$$$$$$$$$$$$$$$$$$$$$$$$$$$$$$$$$$$$$$$$$$$$$$$$$$$$4

\section{Over algebraically closed fields}

In this section, $k$ will denote an algebraically closed field.

\begin{proposition}\label{1.3}
Let $R$ be an affine $k$-algebra of dimension $d$ 
with $1/d!\in k$. Let $f(T)\in R[T]$ be a monic
polynomial. Assume that either

$(i)$ $A=R[T,1/f(T)]$ or

$(ii)$ $A=R[T,f_1/f,\ldots,f_r/f]$, where $f,f_1,\ldots,f_r$ is a
regular sequence in $R[T]$.

Then $A^{d}$ is cancellative.
\end{proposition}

\begin{proof}
$(i)$ Assume that $A=R[T,1/f(T)]$ and let $P$ be a stably free
$A$-module of rank $d$.  Since $A_{1+fk[f]}$ is an affine domain of
dimension $d$ over a $C_1$-field $k(f)$, by Suslin's
result $(\ref{c1})$, $P\ot A_{1+fk[f]}$ is free. Hence, there exists
$h\in 1+fk[f]$ such that $P_h$ is free. By (\cite{murthy00}, Lemma
2.9), patching $P$ and $(R[T]_h)^d$, we get a projective $R[T]$-module
$Q$ of rank $d$ such that $Q_f \iso P$ and $Q_h$ is free. Since $h\in
R[T]$ is a monic polynomial, by (\ref{monic}), $Q$ is free and hence
$P$ is free.  This proves that $A^d$ is cancellative.  \\

$(ii)$ Assume that $A=R[T,f_1/f,\ldots,f_r/f]$ and let $P$ be a
stably free $A$-module of rank $d$. By (\ref{kes3.5}), there exists a
projective $R[T]$-module $Q$ of rank $d$ such that $P\iso Q\ot
A$. Since $P\op A\iso (Q\ot A)\op A$ is free, hence $(Q\op R[T])\ot
R[T,1/f]$ is free. Since $f$ is a monic polynomial, by (\ref{monic}),
$Q\op R[T]$ is free. By (\ref{w002}), $R[T]^d$ is cancellative.  Hence $Q$
is free and therefore $P$ is free. This proves that $A^d$ is
cancellative. $\hfill \square$
\end{proof}

\begin{lemma}\label{z2}
Let $R$ be a reduced ring of dimension $d$ and $A=R[T,1/f(T)]$ for
some $f(T)\in R[T]$. Let $P$ be a projective $A$-module. Then there
exists a non-zero-divisor $s\in R$ satisfying the properties of
(\ref{3.1}).
\end{lemma}

\begin{proof}
Let $S$ be the set of non-zero-divisors of $R$. Then $S^{-1}R$ is a
direct product of fields. Since $K[T,1/g(T)]$ is a PID for any field $K$
and $g(T)\in K[T]$, every projective $K[T,1/g(T)]$-module is
free. Hence every projective module of constant rank over
$S^{-1}R[T,1/f(T)]$ is free. Now, it is easy to see that we can choose
$s\in S$ satisfying the properties of (\ref{3.1}). $\hfill \square$
\end{proof}

\begin{theorem}\label{1.4}
Let $R$ be a reduced affine $k$-algebra of dimension $d$ with $1/d!\in
k$.  Let $f(T)\in R[T]$ be a monic polynomial and let
$A=R[T,1/f(T)]$. Let $P$ be a projective $A$-module of rank $d$. By
(\ref{z2}), choose a non-zero-divisor $s\in R$ satisfying the
properties of (\ref{3.1}).  Then $\Aut(A\op P)$ acts transitively on
$\Um^1(A\op P, s^2A)$.
\end{theorem}

\begin{proof}
Let $C=A[X]/(X^2-s^2X)=B[T,1/f(T)]$, where $B=R[X]/(X^2-s^2X)$ is an
affine $k$-algebra of dimension $d$.  By (\ref{1.3}), $C^d$ is
cancellative. Applying (\ref{q3}), we get that $\Aut(A\op P)$ acts
transitively on $\Um^1(A\op P, s^2A)$. $\hfill \square$
\end{proof}

\begin{remark}
In (\ref{1.4}), if every element of $\Um(A\op P)$ can be taken to an
element of $\Um^1(A\op P, s^2A)$ by an automorphism of $A\op P$, then
$P$ will be cancellative. The same remark is applicable for (\ref{z4},
\ref{z7} and \ref{2.3}).
\end{remark}

\begin{lemma}\label{z3}
Let $R$ be a reduced ring and $A=R[T,f_1/f,\ldots,f_r/f]$ for some
$f,f_1,\ldots,f_r \in R[T]$. Let $P$ be a projective $A$-module with
trivial determinant. Then there exists a non-zero-divisor $s\in R$
satisfying the properties of (\ref{3.1}).
\end{lemma}

\begin{proof}
Let $S$ be the set of non-zero-divisors of $R$. Then $S^{-1}R$ is a
direct product of fields and $\dim S^{-1}A=1$. 
As determinant of $P$ is trivial, by (\ref{serre1}), $S^{-1}P$ is
free.  Now, we can choose $s\in S$
satisfying the properties of (\ref{3.1}). $\hfill \square$
\end{proof}

\begin{remark}
Let $A=K[T,f(T)/g(T)]$, $K$ is a field. We can assume that $f$ and $g$
have no common factors. Hence $(f,g)=K[T]$. Since $A_f=K[T,(fg)^{-1}]$
and $A_g=K[T,g^{-1}]$ are PID, $A$ is a Dedekind domain. We do not
know if all projective $A$-modules are free.
\end{remark}

\begin{theorem}\label{z4}
Let $R$ be a reduced affine $k$-algebra of dimension $d$ with $1/d!\in
k$.  Let $f(T)\in R[T]$ be a monic polynomial and
$A=R[T,f_1/f,\ldots,f_r/f]$, where $f,f_1,\ldots,f_r$ is a regular
sequence in $R[T]$. Let $P$ be a projective $A$-module of rank $d$
with trivial determinant. By (\ref{z3}), choose a non-zero-divisor
$s\in R$ satisfying the properties of $(\ref{3.1})$. Then $\Aut(A\op
P)$ acts transitively on $\Um^1(A\op P,s^2A)$.
\end{theorem}

\begin{proof}
Let $C=A[X]/(X^2-s^2X)=B[T,f_1/f,\ldots,f_r/f]$, where
$B=R[X]/(X^2-s^2X)$ is an affine $k$-algebra of dimension $d$. Since
$B$ is a free $R$-module, $f,f_1,\ldots,f_r$ is a $B[T]$-regular
sequence. By (\ref{1.3}), $C^d$ is cancellative. Applying (\ref{q3}),
we get that $\Aut(A\op P)$ acts transitively on $\Um^1(A\op P,
s^2A)$. $\hfill \square$
\end{proof}

\begin{theorem}\label{fp}
Let $R$ be an affine $\ol \BF_p$-algebra of dimension $d\geq 3$, where
$p>d$. Let $f(T)\in R[T]$ be a monic polynomial and
$A=R[T,f_1/f,\ldots,f_r/f]$ for some $f,f_1,\ldots,f_r\in R[T]$.  Then
every projective $A$-module of rank $d$ with trivial determinant is
cancellative.
\end{theorem}

\begin{proof}
First we prove that $A^d$ is cancellative. Let $P$ be a stably free
$A$-module of rank $d$. By Suslin's result (\ref{z1}(i)), we may
assume that $P\op A$ is free. By (\cite{kes05}, Theorem 3.6), $P$ is
extended from $R[T]$. Now, we can complete the proof as in (\ref{1.3}(ii)).

Let $P$ be a projective $A$-module of rank $d$ with trivial determinant.
We may assume that $A$ is reduced.  By (\ref{z3}), choose a
non-zero-divisor $s\in R$ satisfying the properties of $(\ref{3.1})$.
If $C=A[Y]/(Y^2-s^2Y)$, then, as in the previous paragraph, $C^d$ is
cancellative. By (\ref{q3}), $\Aut(A\op P)$ acts transitively on
$\Um^1(A\op P,s^2A)$. Applying (\ref{kumar}) and (\ref{trans}), it is
easy to see that every element of $\Um(A\op P)$ can be taken to an
element of $\Um^1(A\op P, s^2A)$ by an automorphism of $A\op P$. This
proves that $P$ is cancellative.
$\hfill \square$
\end{proof}

As a consequence of (\ref{fp}), we get the following result which
extends a result of Murthy (\cite{murthy00}, Corollary 2.13), where it
is proved that $A^{d}$ is cancellative.

\begin{theorem}
Let $R=\ol \BF_p[X_1,\ldots,X_{d+1}]$ and let $A$ be a subring of the
fraction field of $R$ with $R\subset A$. Suppose $p>d\geq 3$. Then all
projective $A$-modules of rank $d$ with trivial determinant are cancellative.
\end{theorem}

Using (\ref{1.3}) and following the proofs of (\cite{murthy00},
Proposition 3.1 and Theorem 3.6), we get the following two results.

\begin{corollary}
Let $1/(d-1)!\in k$ and
$A=k[x_0,x_1,\ldots,x_d]$, where $x_0^2+x_1^2 + f(x_2,\ldots,x_d)=0$
for some $f\in k[x_2,\ldots,x_d]$. Then $A^{d-1}$ is cancellative
\end{corollary}

\begin{corollary}
Let $1/(d-1)!\in k$ and
$A=k[x,y,t_1,\ldots,t_{d-1}]$, where
$x+x^sg(x,t_1,\ldots,t_{d-1})+x^ry+f(t_1,\ldots,t_{d-1})=0$, with
$d\geq 3$, $s\geq 2$, $r\geq 2$.  Then $A$ is a smooth $d$ dimensional
affine $k$-algebra. Further $A^{d-1}$ is cancellative.
\end{corollary}

%$$$$$$$$$$$$$$$$$$$$$$$$$$$$$$$$$$$$$$$$$$$$$$$$$$
%####################################################

\section{Over real closed fields}

In this section, $k$ will denote a real closed field.

\begin{proposition}\label{z5}
Let $R$ be an affine $k$-algebra of dimension $d-1\geq 2$ and let $f(T)\in
R[T]$ be a monic polynomial. Assume that $f(T)$ does not belongs to any
real maximal ideal of $R[T]$ and either

$(i)$ $A=R[T,1/f(T)]$ or

$(ii)$ $A=R[T,f_1/f,\ldots,f_r/f]$, where $f,f_1,\ldots,f_r$ is a
regular sequence in $R[T]$.

Then $A^d$ is cancellative.
\end{proposition}

\begin{proof}
$(i)$ Let $A=R[T,1/f(T)]$ and let $P$ be a stably free $A$-module of
rank $d$. Then $P\op A$ is free, by (\ref{bass}). By (\cite{parimala},
Theorem), $P$ is extended from $R[T]$. Let $Q$ be a projective
$R[T]$-module such that $P=Q\ot A$. Then $(Q\op R[T])_f$ is free and
$f$ is a monic polynomial, hence $Q\op R[T]$ is free, by
(\ref{monic}). By Plumstead's result \cite{plumstead}, every
projective $R[T]$-module of rank $> \dim R$ is cancellative. Hence $Q$
is free and therefore $P$ is free.\\

$(ii)$ Let $A=R[T,f_1/f,\ldots,f_r/f]$ and let $P$ be a stably free
$A$-module of rank $d$. Then $P\op A$ is free, by (\ref{bass}).  By
(\ref{kes4.4}), $P$ is extended from $R[T]$. Let $Q$ be a projective
$R[T]$-module of rank $d$ such that $P\iso Q\ot A$.  Since $(Q\ot
A)\op A$ is free, $(Q\op R[T])\ot R[T,1/f]$ is free. As $f$ is a monic
polynomial, by (\ref{monic}), $Q\op R[T]$ is free. By Plumstead's
result \cite{plumstead}, $Q$ is cancellative.  Hence $Q$ is free and
so $P$ is free.  $\hfill \square$
\end{proof}

\begin{theorem}\label{z6}
Let $R$ be an affine $k$-algebra of dimension $d-1\geq 2$ and let
$f(T)\in R[T]$ be a monic polynomial. Assume that $f(T)$ does not belong
to any real maximal ideal of $R[T]$ and $A=R[T,1/f(T)]$.
Then every projective $A$-module of rank $d$ is cancellative.
\end{theorem}

\begin{proof}
We may assume that $R$ is reduced.  Let $P$ be a projective $A$-module
of rank $d$. By (\ref{z2}), we can
choose a non-zero-divisor $s\in R$ satisfying the properties of
(\ref{3.1}).  Let $C=A[X]/(X^2-s^2X) =B[T,1/f(T)]$, where
$B=R[X]/(X^2-s^2X)$. Since $B[T]$ is a finite extension of $R[T]$,
any maximal ideal of $B[T]$ will contract to a maximal ideal of
$R[T]$. Therefore, $f(T)$ does not belongs to any real maximal
ideal of $B[T]$. By (\ref{z5}), $C^d$ is cancellative. Hence, by
(\ref{q4}), $P$ is cancellative.  $\hfill \square$.
\end{proof}
 
\begin{theorem}\label{2.2}
Let $R$ be an affine $k$-algebra of dimension $d-1\geq 2$ and let
$f(T)\in R[T]$ be a monic polynomial. Assume that $f(T)$ does not belong
to any real maximal ideal of $R[T]$ and $A=R[T,f_1/f,\ldots,f_r/f]$,
where $f,f_1,\ldots,f_r$ is a regular sequence in $R[T]$.  Then every
projective $A$-module of rank $d$ with trivial determinant is
cancellative.
\end{theorem}

\begin{proof}
We may assume that $R$ is reduced.  Let $P$ be a projective $A$-module
of rank $d$ with trivial determinant. By (\ref{z3}), we can choose a
non-zero-divisor $s\in R$ satisfying the properties of (\ref{3.1}).
Let $C=A[X]/(X^2-s^2X) =B[T,f_1/f,\ldots,f_r/f]$, where
$B=R[X]/(X^2-s^2X)$. Since $B[T]$ is a finite extension of $R[T]$, any
maximal ideal of $B[T]$ will contract to a maximal ideal of
$R[T]$. Therefore, $f(T)$ does not belongs to any real maximal ideal
of $B[T]$. Also, since $B[T]$ is a free $R[T]$-module,
$f,f_1,\ldots,f_r$ is a regular sequence in $B[T]$.  By (\ref{z5}),
$C^d$ is cancellative. Hence, by (\ref{q4}), $P$ is cancellative.
$\hfill \square$.
\end{proof}

\begin{proposition}\label{free2}
Let $R$ be an affine $k$-algebra of dimension $d-2\geq 1$. Let
$A=R[X,T,1/f]$, where $f\in R[X,T]$ is a monic polynomial in $T$ and
$f$ does not belong to any real maximal ideal of $R[X,T]$. Then $A^{d-1}$
is cancellative.
\end{proposition}

\begin{proof}
Let $P$ be a stably free $A$-module of rank $d-1$. By (\ref{z5}), we
may assume that $P\op A$ is free. By (\ref{kes4.4}), $P$ is extended
from $R[X,T]$. 
Let $Q$ be a projective $R[X,T]$-module such that $P\iso Q\ot
A$. Since $(Q\op R[X,T])\ot A$ is free and
$f$ is a monic polynomial, by (\ref{monic}), $Q \op
R[X,T]$ is free. By Ravi Rao's result (\cite{ravi1}, Theorem 2.5),
every projective $R[X_1,\ldots,X_n]$-module of rank $>\dim R$ is
cancellative. Hence $Q$ is free and therefore $P$ is free. $\hfill \square$
\end{proof}

\begin{corollary}
Let $A=k[X_1,\ldots,X_d,1/f]$. Assume that $f$ does not belongs to any
real maximal ideal of $k[X_1,\ldots,X_d]$. Then every projective
$A$-module of rank $\geq d-1$ is free.
\end{corollary}

\begin{proof}
Since $K_0(A)=\BZ$, every projective $A$-module is stably
free. Now the result follows from (\ref{free2}). $\hfill \square$
\end{proof}

\begin{theorem}\label{z7}
Let $R$ be an affine $k$-algebra of dimension $d-2\geq 1$. Let
$A=R[X,T,1/f]$, where $f\in R[X,T]$ is a monic polynomial in $T$ and
$f$ does not belong to any real maximal ideal of $R[X,T]$. Let $P$ be
a projective $A$-module of rank $d-1$. Assume that there exists a
non-zero-divisor $s\in R$ satisfying the properties of
(\ref{3.1}). Then $\Aut(A\op P)$ acts transitively on $\Um^1(A\op
P,s^2A)$. 
\end{theorem}

\begin{proof}
Let $C=A[Y]/(Y^2-s^2Y) = B[X,T,1/f]$, where $B=R[Y]/(Y^2-s^2Y)$. Since
$B[X,T]$ is a finite extension of $R[X,T]$, $f$ does not belongs to
any real maximal ideal of $B[X,T]$. By (\ref{free2}), $C^{d-1}$ is
cancellative. By (\ref{q3}), $\Aut(A\op P)$ acts transitively on
$\Um^1(A\op P, s^2A)$.  $\hfill \square$
\end{proof}

\begin{proposition}\label{free1}
Let $R$ be an affine $k$-algebra of dimension $d-2\geq 1$. Let
$A=R[X,T,f_1/f,\ldots,f_r/f]$, where $f,f_1,\ldots,f_r$ is a
$R[X,T]$-regular sequence. Assume that $f$ is a monic polynomial in
$T$ and $f$ does not belong to any real maximal ideal of
$R[X,T]$. Then $A^{d-1}$ is cancellative.
\end{proposition}

\begin{proof}
Let $P$ be a stably free $A$-module of rank $d-1$. By (\ref{z5}),
$P\op A$ is free and hence by (\ref{kes4.4}), $P$ is extended from
$R[X,T]$. Let $Q$ be a projective $R[X,T]$-module such that $P\iso Q\ot
A$. Since $(Q\op R[X,T])\ot A$ is free, $(Q\op R[X,T])\ot R[X,T]_f$ is
free. Since $f$ is a monic polynomial, by (\ref{monic}), $Q \op
R[X,T]$ is free. By Ravi Rao's result (\cite{ravi1}, Theorem 2.5),
every projective $R[X_1,\ldots,X_n]$-module of rank $>\dim R$ is
cancellative. Hence $Q$ is free and therefore $P$ is free. $\hfill \square$
\end{proof}

\begin{theorem}\label{2.3}
Let $R$ be an affine $k$-algebra of dimension $d-2\geq 1$. Let
$A=R[X,T,f_1/f,\ldots,f_r/f]$, where $f,f_1,\ldots,f_r$ is a
$R[X,T]$-regular sequence. Assume that $f$ is a monic polynomial in
$T$ and $f$ does not belong to any real maximal ideal of $R[X,T]$. Let
$P$ be a projective $A$-module of rank $d-1$. Assume that there exists
a non-zero-divisor $s\in R$ satisfying the properties of (\ref{3.1}).
Then $\Aut(A\op P)$ acts transitively on $\Um^1(A\op P,s^2A)$.
\end{theorem}

\begin{proof}
Let $C=A[Y]/(Y^2-s^2Y) = B[X,T,f_1/f,\ldots,f_r/f]$, where
$B=R[Y]/(Y^2-s^2Y)$. Since $B$ is a finite extension of $R$, every
maximal ideal of $B[X,T]$ will contract to a maximal ideal of
$R[X,T]$. Hence $f$ does not belongs to any real maximal ideal of
$B[X,T]$. Also, as $B[X,T]$ is a free module over $R[X,T]$,
$f,f_1,\ldots,f_r$ is a regular sequence in $B[X,T]$. By
(\ref{free1}), $C^{d-1}$ is cancellative. Hence, by (\ref{q3}),
$\Aut(A\op P)$ acts transitively on $\Um^1(A\op P, s^2A)$.
$\hfill \square$
\end{proof}

%$$$$$$$$$$$$$$$$$$$$$$$$$$$$$$$$$$$$$$$$$$$$$$$$$$$$
%&&&&&&&&&&&&&&&&&&&&&&&&&&&&&&&&&&&&&&&&&&&&&&&&&&&&7

\section{Over Laurent polynomial rings}

\begin{lemma}\label{z8}
Let $R$ be a reduced ring of dimension $d$ and let
$A=R[X_1,\ldots,X_n,Y_1^{\pm 1},\ldots,Y_m^{\pm 1}]$.  Let $P$ be a
projective $A$-module. Then there exists a non-zero-divisor $s\in R$
satisfying the properties of (\ref{3.1}).
\end{lemma}

\begin{proof}
Let $S$ be the set of non-zero-divisors of $R$. Then $S^{-1}R$ is a
direct product of fields. Suslin (\cite{Su1}, Corollary 7.4) and Swan
(\cite{Swan}, Theorem 1.1) independently proved that if $K$ is a field
or a PID, then every projective $K[X_1,\ldots,X_n,Y_1^{\pm
1},\ldots,Y_m^{\pm 1}]$-modules are free. Hence $S^{-1}P$ is free. Now
we can choose $s\in S$ satisfying the properties of (\ref{3.1}).
$\hfill\square$
\end{proof}

\begin{theorem}\label{1.1}
Let $R$ be a ring of dimension $d$ and let
$A=R[X_1,\ldots,X_n,Y_1^{\pm 1},\ldots,Y_m^{\pm 1}]$.  Let $P$ be a
projective $A$-module of rank $\geq d$. By (\ref{z8}), choose a
non-zero-divisor $s\in R$ satisfying the properties of
(\ref{3.1}). Assume that $B^{d}$ is cancellative, where
$B=A[T]/(T^2-s^2T)$.  Then $P$ is cancellative.
\end{theorem}

\begin{proof}
By (\ref{q3}), $\Aut(A\op P)$ acts transitively on $\Um^1(A\op
P,s^2A)$. Let ``bar'' denotes reduction modulo $s^2A$. Since $\dim
R/s^2R <d$, by (\ref{lin}), $E(\ol A\op \ol P)$ acts transitively on
$\Um(\ol A\op \ol P)$. Further, by (\ref{trans}), every element of
$E(\ol A\op \ol P)$ can be lifted to an element of $\Aut(A\op
P)$. Hence, every element of $\Um(A\op P)$ can be taken to an element
of $\Um^1(A\op P,s^2A)$ by an automorphism of $A\op P$. This proves
that $P$ is cancellative.  $\hfill \square$
\end{proof}

\begin{theorem}\label{z9}
Let $R$ be an affine algebra of dimension $d$ over an algebraically
closed field $k$ with $1/d!\in k$. Let $A=R[X,X^{-1}]$. Then
every projective $A$-module of rank $d$ is cancellative.
\end{theorem}

\begin{proof}
We can assume that $A$ is reduced. Let $P$ be a projective $A$-module
of rank $d$. By (\ref{z8}), choose a non-zero-divisor $s\in R$
satisfying the properties of (\ref{3.1}). Let
$B=A[T]/(T^2-s^2T)=B_1[X,X^{-1}]$, where $B_1=R[T]/(T^2-s^2T)$ is an
affine algebra over $k$ of dimension $d$. By (\ref{1.3}), $B^d$ is
cancellative. Hence, applying (\ref{1.1}), we get that
$P$ is cancellative. This proves the result.  $\hfill \square$
\end{proof}

\begin{theorem}\label{z10}
Let $R$ be a ring of dimension $d$. Let $A=R[X_1,\ldots,X_n,Y_1^{\pm
1},\ldots,Y_m^{\pm 1}]$ and let $P$ be a projective $A$-module of rank
$\geq d$. By (\ref{z8}), choose a non-zero-divisor $s\in R$ satisfying
the properties of (\ref{3.1}). Assume that $B^{d}$ is cancellative,
where $B=B_1[X_1,\ldots,X_n,Y_1^{\pm 1},\ldots,Y_m^{\pm 1}]$ and
$B_1=R[T]/(T^2-s^2T)$.  Then the natural map $\Aut (P) \ra
\Aut_{A/(Y_m-1)} (P/(Y_m-1)P)$ is surjective.
\end{theorem}

\begin{proof}
When rank $P> d$, the result follows from (\ref{3.2}). Hence, we
assume that rank $P=d$. Let ``bar'' denotes reduction modulo
$(Y_m-1)A$. It is easy to see that we can assume that $R$ is reduced.

Let $\tau\in \Aut_{\ol A}(\ol P)$, then by (\ref{3.2}), we can lift
$id_{\ol A}\op \tau \in \Aut_{\ol A}(\ol A\op \ol P)$ to an
automorphism $\theta$ of $A\op P$. Let $\theta (1,0)=(h,p) \in
\Um(A\op P, Y_m-1)$. Assume that there exists $\mu \in \Aut(A\op P,
Y_m-1)$ such that $\mu (h,p)=(1,0)$. Then, we have the following
commutative diagram

$$\xymatrix{ A \op P \ar [r]^\theta \ar [d] & A \op P \ar [r]^\mu \ar [d]
& A\op P\ar [d] \\ \ol A \op \ol P \ar [r]_{id_{\ol A} \op \sigma} & \ol
A \op \ol P \ar [r]^{id} & \ol A\op \ol P }$$

Note that $\Psi=\mu\theta \in \Aut (A\op P)$ is a lift of $id_{\ol
A}\op \tau$. Further $\Psi(1,0)=(1,0)$. Hence $\Psi$ induces an
automorphism $\Psi\in \Aut (P)$ which is a lift of $\tau$. Hence, it
is enough to show that $\Aut(A\op P, Y_m-1)$ acts transitively on
$\Um(A\op P,Y_m-1)$.

Let $(f,q)\in \Um(A\op P,Y_m-1)$. 
Let ``tilde'' denote reduction modulo $s^3A$.  Since $\dim R/s^3R < \dim
R$, by (\ref{3.2}), $EL^1( \wt A\op \wt P,Y_m-1)$
acts transitively on $\Um^1(\wt A\op \wt P,Y_m-1)$. After lifting the
$EL^1(\wt A\op \wt P,Y_m-1)$ transformations, we may assume that
$(f,q)=(1,0)$ modulo $s^3 (Y_m-1)A$. 

Since $q\in s^3 (Y_m-1)P$, with the notation in (\ref{3.1}), we can
write $q=f_1p_1+\ldots+f_dp_d$ for some $f_i\in s^2 (Y_m-1)A$,
$i=1,\ldots,d$. Note that $(f,f_1,\ldots,f_d)\in \Um_{d+1}(A,s^2(Y_m-1))$.

Since $B^d$ is cancellative, where $B$ is the cartesian square of
$A,A$ over $A/(s^2)$, as in the proof of (\ref{q3}), there exists
$\Delta\in \GL_{d+1}(A,s^2A)$ such that
$(f,f_1,\ldots,f_d)\Delta=(1,0,\ldots,0)$. Since
$(f,f_1,\ldots,f_d)\in \Um_{d+1}(A,s^2(Y_m-1))$, going modulo $(Y_m-1)$, we get
$(1,0,\ldots,0)\Delta(Y_m-1)=(1,0,\ldots,0)$. Hence, if $\Theta=\Delta
\Delta(Y_m-1)^{-1}$, then $(f,f_1,\ldots,f_d)\Theta=(1,0,\ldots,0)$
and $\Theta\in \GL_{d+1}(A,s^2(Y_m-1))$.

Let $\Delta=1+TN^2$, where $T$ is some matrix and $N=$ diagonal
$(1,s,\ldots,s)$. Applying $(\ref{w001})$ with $n=d+1$ and
$(s_1,\ldots,s_n)=(1,s,\ldots,s)$, we get $\Psi=\Phi(\delta) \in \Aut (A\op
P,Y_m-1)$ and $\Psi(f,f_1p_1+\ldots+f_dp_d)=(1,0)$. This proves the result.
$\hfill \square$
\end{proof}

As an application of (\ref{z9}) and (\ref{z10}), we get the following
result.

\begin{corollary}\label{z11}
Let $R$ be an affine algebra of dimension $d$ over an algebraically
closed field $k$ with $1/d!\in k$. Let $A=R[X,X^{-1}]$ and let $P$ be
a projective $A$-module of rank $d$. Then the natural map $\Aut(P) \ra
\Aut_{A/(Y-1)}(P/(Y-1)P)$ is surjective.
\end{corollary}

We end this section by stating four results which follow directly from
(\ref{w002}) by applying (\ref{c1}), (\ref{z9}, \ref{z10}),
(\ref{z1}(iii)) and (\ref{a1}) respectively. Note that (\ref{gen}(i))
generalizes a result of Keshari (\cite{kes07}, Proposition A.9), where
it is proved when $A$ is a smooth affine algebra of dimension $d=2$
and the determinant of $P$ is trivial.

\begin{theorem}
Let $A$ be an affine algebra of dimension $d$ over a $C_1$-field
$k$ with $1/d!\in A$ and $R=A[X_1,\ldots,X_n]$.  Then every
projective $R$-module of rank $d$ is cancellative.
\end{theorem}

\begin{theorem}\label{gen}
Let $A$ be an affine algebra of dimension $d$ over an algebraically
closed field $k$ with $1/d!\in k$ and $R=A[X_1,\ldots,X_n,Y^{\pm
1}]$. Let $P$ be a projective $R$-module of rank $d$. Then 

$(i)$ $P$ is cancellative and 

$(ii)$ the natural map $\Aut(P) \ra\Aut(P/(Y-1)P)$ is surjective.
\end{theorem}

\begin{theorem}
Let $A$ be a finitely generated algebra over $\BZ$ of dimension
$d$ with $1/d! \in A$. Then all projective $A[X_1,\ldots,X_n]$-modules
of rank $d$ are cancellative.
\end{theorem}

\begin{theorem}
Let $k\subset \ol \BF_p$ be a field and let $A$ be an affine
$k$-algebra of dimension $d$. Assume that $p >d$.  Then all
projective $A[X_1,\ldots,X_n]$-modules of rank $d$ are cancellative.
\end{theorem}
\vspace*{.1in}

\noindent{\bf Acknowledgments.}
I sincerely thank S.M. Bhatwadekar 
for useful discussion on (\ref{cartesian1}).

{\small
%&&&&&&&&&&&&&&&&&&&&&&&&&&&&&&&&&&&&&&&&&&&&&&&&&&&&&&
%----------------------------------------------

}

\end{document}